\documentclass{amsart}

\usepackage{amsmath}
\usepackage{amssymb}
\usepackage{graphicx}
\usepackage{amsthm}
\usepackage{tikz}
\usepackage{enumerate}

\usepackage{amsmath}
\usepackage{amsfonts}
\usepackage{mathrsfs}
\usetikzlibrary{arrows}
\usepackage[all]{xy}

\theoremstyle{plain}
\newtheorem{theorem}{Theorem}[section]
\newtheorem{proposition}[theorem]{Proposition}
\newtheorem{lemma}[theorem]{Lemma}
\newtheorem{corollary}[theorem]{Corollary}

\theoremstyle{definition}
\newtheorem{definition}[theorem]{Definition}
\newtheorem{remark}[theorem]{Remark}

\numberwithin{equation}{section}

\usepackage[hidelinks]{hyperref}
\hypersetup{colorlinks=true,citecolor=blue,linkcolor=blue}

\newcommand{\Down}{\mathcal{D}}

\newcommand{\A}{\mathbf{A}}
\renewcommand{\L}{\mathbf{L}}						
\newcommand{\M}{\mathbf{M}}

\newcommand{\Fi}{\mathrm{Fi}}					
\newcommand{\Fig}{\mathrm{Fig}}					

\newcommand{\B}{\mathrm{Bo}}					
\newcommand{\At}{\mathrm{At}_{\mathrm{d}}}		

\newcommand{\D}{\mathrm{D}}						

\newcommand{\comp}{\mathrm{Comp}}					
\newcommand{\C}{\mathcal{C}}						

\newcommand{\Irr}{\mathrm{Irr}}					

\newcommand{\N}{\mathbf{N}}						
\renewcommand{\S}{\mathbf{S}}						

\begin{document}

\title{Finite distributive nearlattices}

\address{Universidad Nacional de La Pampa. Facultad de Ciencias Exactas, y Naturales. Santa Rosa, Argentina.}
\email{lucianogonzalez@exactas.unlpam.edu.ar}
\author{Luciano J. Gonz\'{a}lez}

\address{CIC and Departamento de Matem\'{a}tica, Facultad de Ciencias Exactas, Universidad Nacional del Centro, Pinto 399, 7000 Tandil, Argentina.}
\email{calomino@exa.unicen.edu.ar}
\author{Ismael Calomino}


\begin{abstract}
Our main goal is to develop a representation for finite distributive nearlattices through certain ordered structures. This representation generalizes the well-known representation given by Birkhoff for finite distributive lattices through finite posets. We also study finite distributive nearlattices through the concepts of dual atoms, boolean elements, complemented elements and irreducible elements. We prove that the sets of boolean elements and complemented elements form semi-boolean algebras. We show that the set of boolean elements of a finite distributive lattice is a boolean lattice.
\end{abstract}

\maketitle

\section{Introduction and preliminaries}

It is well known that the class of distributive lattices have many and important applications,  especially in logic and computer science. Thus, it is also important to study some natural generalizations of distributive lattices which may result interesting. An important class of distributive lattices for computer science is the class of finite distributive lattices. This paper deals with the concept of \emph{finite distributive nearlattice}, which is a natural and nice generalization of finite distributive lattice.

Distributive nearlattices were study from different points of view: algebraic, topological and logical \cite{ChaKo08,Ha06,ChaKo07,ArKi11,CaCe15,CeCa16,GoCa19,CaCeGo19,CeCa19,Ca19,CaGo20,CeCa14,Go18,Go19}.

In Section \ref{sec:atoms and boolean elements} we consider the notions of dual atom and boolean element. We prove that the boolean elements form a semi-boolean algebra (\cite{Abb67}). In Section \ref{sec:annihilators} we show that the semi-boolean algebra of boolean elements of a finite distributive nearlattice $\A$ is a homomorphic image of $\A$. We generalize the concept of complement elements from distributive lattice to distributive nearlattice and, we show that the set of complement elements of a distributive nearlattice is a semi-boolean algebra. Section \ref{sec: free extension} is concerned with the free distributive lattice extension of a distributive nearlattice (\cite{CeCa16}). Finally, in Section \ref{sec:representation} we develop a discrete representation for the class of finite distributive nearlattice, which is a generalization of the well-known representation for finite distributive lattice given by Birkhoff \cite{Bi37}.

We assume that the reader is familiar with the theory of ordered sets and lattices. Our main references for Order and Lattice theory are \cite{DaPri02,Gra11,BaDwi74}. Now we introduce the notational conventions that we use throughout the paper.

Let $P=\langle P,\leq\rangle$ be a poset. A subset $U\subseteq P$ is said to be an \emph{upset} when for all $a,b\in P$, if $a\leq b$ and $a\in U$, then $b\in U$. Dually we have the notion of \emph{downset}. Given an element $a\in P$, the \emph{principal upset} generated by $a$ is $[a):=\{x\in P: a\leq x\}$. Dually we have the \emph{principal downset} $(a]$. Given a subset $Q$ of $P$ and $a\in P$, $[a)_Q:=\{x\in Q: a\leq x\}$.

\begin{definition}[\cite{ChaKo08}]
A \emph{distributive nearlattice} is a join-semilattice $\A=\langle A,\vee,1\rangle$ with top element 1 such that for all $a\in A$, the principal upset $[a)$ is a bounded distributive lattice (with respect to the induced order).
\end{definition}

Let $\A=\langle A,\vee,1\rangle$ be a distributive nearlattice. For each $a\in A$, we denote the meet in $[a)$ by $\wedge_a$. It should be noted that for all $x,y\in A$, the meet $x\wedge y$ exists in $A$ if and only if $x,y$ have a common lower bound in $A$. Thus, for all $x,y\in[a)$, the meet of $x,y$ in $[a)$ coincides with their meet in $A$, that is, $x\wedge_ay=x\wedge y$. This should be kept in mind since we will use it without mention. A subset $F\subseteq A$ is called \emph{filter} if (i) $1\in F$; (ii) $F$ is an upset; and (iii) if $a,b\in F$ and $a\wedge b$ exists in $A$, then $a\wedge b\in F$. Let us denote by $\Fi(\A)$ the collection of all filters of $\A$. It is straightforward that $\Fi(\A)$ is a closure system. For $X\subseteq A$, we denote by $\Fig_A(X)$ the filter generated by $X$. Thus, we have that $\langle\Fi(\A),\cap,\veebar,\{1\},A\rangle$ is a distributive lattice where $F\veebar G=\Fig_A(F\cup G)$, for all $F,G\in\Fi(\A)$.

A non-empty subset $I\subseteq A$ is said to be an \emph{ideal} when (i) $I$ is a downset; (ii) if $a,b\in I$, then $a\vee b\in I$. A proper ideal $I$ of $A$ is called \emph{prime} when for all $a,b\in A$, if $a\wedge b$ exists and $a\wedge b\in I$, then $a\in I$ or $b\in I$.

\begin{theorem}[\cite{Ha06}]
Let $\A$ be a distributive nearlattice. If $I$ is an ideal and $F$ is a filter of $\A$ such that $I\cap F=\emptyset$, then there exists a prime filter $P$ of $A$ such that $I\subseteq P$ and $P\cap F=\emptyset$.
\end{theorem}

Given two distributive nearlattices $\A_1$ and $\A_2$, a map $f\colon A_1\to A_2$ is said to be an \emph{N-homomorphism} if $f(1)=1$; for all $a,b\in A_1$, $f(a\vee b)=f(a)\vee f(b)$, and $f(a\wedge b)=f(a)\wedge f(b)$, whenever $a\wedge b$ exists in $A_1$.

The reader is referred to \cite{ChaHaKu07,ChaKo08,Ca15} for further information about distributive nearlattices.

\section{Dual atoms and boolean elements}\label{sec:atoms and boolean elements}

\begin{definition}
Let $\A$ be a distributive nearlattice and $a \in A$. We will say that $a$ is a {\it{dual atom}} if $a \neq 1$ and for each $x \in A$,  $a \leq x \leq 1$ implies $x=a$ or $x=1$. We say that $a$ is a {\it{boolean element}} if the principal filter $[a)$ is a boolean lattice.
\end{definition}

Let us denote by $\At(\A)$ and $\B(\A)$ the collections of all dual atoms of $\A$ and of all boolean elements of $\A$, respectively. It follows that $\At(\A)\subseteq\B(\A)$. If $a\in\B(\A)$ and $x\in[a)$, we shall denote the complement of $x$ in $[a)$ by $\neg_ax$. The next properties will be used throughout the paper without mention.

\begin{remark}
Let $a\in\At(\A)$. Then, for all $x\in A$, we have that $x\leq a$ or $x\vee a=1$. Moreover, if $\A$ is a finite distributive nearlattice, then for every $x<1$, there exists $a\in\At(\A)$ such that $x\leq a$.
\end{remark}

\begin{proposition}\label{prop: meet of atoms are booleans}
Let $\A$ be a distributive nearlattice. Let $a_1,\dots,a_n\in\At(\A)$ be such that $a_1\wedge\dots\wedge a_n$ exists in $A$. Then $a_1\wedge\dots\wedge a_n\in\B(\A)$.
\end{proposition}

\begin{proof}
Let $a^*=a_1\wedge\dots\wedge a_n$. Since $\langle[a^*),\wedge_{a^*},\vee,a^*,1\rangle$ is a bounded distributive lattice, it only remains to verify that each element of $[a^*)$ has a complement. For each $a_i$, let us denote the complement of an element $x\in[a_i)$ by $\neg_i x$. That is, for every $x\in[a_i)$, $\neg_i x\in[a_i)$ such that $x\wedge_{a_i}\neg_ix=a_i$ and $x\vee\neg_i x=1$.

Let $x\in[a^*)$. Thus $x=(a_1\vee x)\wedge\dots\wedge(a_n\vee x)$. We define the element $\neg x:=\neg_1(a_1\vee x)\wedge\dots\wedge\neg_n(a_n\vee x)$ (notice that the last meet exists because $a^*\leq a_i\leq\neg_i(a_i\vee x)$, for all $i$). It is clear that $\neg x\in[a^*)$. Now we show that $\neg x$ is the complement of $x$ in $[a^*)$. On the one hand, we have
\begin{align*}
\neg x\wedge x &= \bigwedge_{1\leq i\leq n}\neg_i(a_i\vee x) \wedge \bigwedge_{1\leq i\leq n}(a_i\vee x)
=
\bigwedge_{1\leq i\leq n}[\neg_i(a_i\vee x)\wedge(a_i\vee x)]\\
&=
\bigwedge_{1\leq i\leq n}a_i =a^*.
\end{align*}
On the other hand, we have
\[
\neg x\vee x =\left(\bigwedge_{1\leq i\leq n}\neg_i(a_i\vee x)\right)\vee x = \bigwedge_{1\leq i\leq n}\left(\neg_i(a_i\vee x)\vee x\right).
\]
Since every $a_i$ is a dual atom, it follows that
\[
\neg_i(a_i\vee x) =
\begin{cases}
1 \ &\text{ if } \ x\leq a_i\\
a_i   \ &\text{ if } \ x\nleq a_i
\end{cases}
\]
Hence, in any case, we obtain that $\neg_i(a_i\vee x)\vee x=1$. Then $\neg x\vee x=1$. We have proved that $\neg x$ is the complement of $x$ in $[a^*)$. Therefore, $\langle[a^*),\wedge_{a^*},\vee,a^*,1\rangle$ is a boolean lattice, and thus $a^*=a_1\wedge\dots\wedge a_n\in\B(\A)$.
\end{proof}

Let $\A$ be a distributive nearlattice and $a\in A$. We define the following set:
\[
X_a:=\{x\in\At(\A): a\leq x\}. 
\]
Thus, $X_a=\At(\A)\cap[a)=\At([a))$. If $\A$ is finite and $a\in\B(\A)$, then since $[a)$ is a finite boolean lattice it follows that
\[
a=\bigwedge X_a.
\]

\begin{lemma}\label{lem: aux 1}
Let $\A$ be a distributive nearlattice and $a,b\in A$. Then $X_{a\vee b}=X_a\cap X_b$. If $a\wedge b$ exists, then $X_{a\wedge b}=X_a\cup X_b$.
\end{lemma}

\begin{proof}
It is straightforward to show directly that $X_{a\vee b}=X_a\cap X_b$.

It is clear that $X_a\cup X_b\subseteq X_{a\wedge b}$. Let now $x\in X_{a\wedge b}$. Since $a\wedge b\leq x$, we have that $x=(a\vee x)\wedge(b\vee x)$. Given that $x\in\At(\A)$, it follows that $(a\leq x \text{ or } a\vee x=1)$ and $(b\leq x \text{ or } b\vee x=1)$. If $a\vee x=1$ and $b\vee x=1$, then $x=1$, which is a contradiction. Hence, $a\leq x$ or $b\leq x$. That is, $x\in X_a\cup X_b$. Therefore $X_{a\wedge b}\subseteq X_a\cup X_b$.
\end{proof}

\begin{proposition}
Let $\A=\langle A,\vee,1\rangle$ be a finite distributive nearlattice. Then $\langle\B(\A),\vee,1\rangle$ is a nearlattice subalgebra of $\A$.
\end{proposition}

\begin{proof}
We need to prove that $\B(\A)$ is closed under $\vee$, and if $a,b\in\B(\A)$ are such that $a\wedge b$ exists in $A$, then $a\wedge b\in\B(\A)$.

Let $a,b\in\B(\A)$. Notice that $a\vee b\leq\bigwedge X_{a\vee b}$, and since $a,b\in\B(\A)$, it follows that $a=\bigwedge X_a$ and $b=\bigwedge X_b$. By Proposition \ref{prop: meet of atoms are booleans}, we have $\bigwedge X_{a\vee b}\in\B(\A)$. Let us show that $a\vee b=\bigwedge X_{a\vee b}$. Suppose, towards a contradiction, that $a\vee b\neq\bigwedge X_{a\vee b}$. So $\bigwedge X_{a\vee b}\nleq a\vee b$. Then, there is a prime ideal $P$ such that $a\vee b\in P$ and $\bigwedge X_{a\vee b}\notin P$. Thus, $a,b\in P$ and $P\cap X_{a\vee b}=\emptyset$. Since $\bigwedge X_a\in P$, $\bigwedge X_b\in P$ and $P$ is prime, it follows that there is $x_a\in X_a$ such that $x_a\in P$ and there is $x_b\in X_b$ such that $x_b\in P$. Since $x_a$ and $x_b$ are dual atoms, we have 
\[
x_a\vee x_b=
\begin{cases}
1 \ &\text{ if } \ x_a\neq x_b\\
x_a=x_b \ &\text{ if } \ x_a=x_b
\end{cases}
\]
Since $P$ is proper and $x_a\vee x_b\in P$, we have that $x_a=x_b$. Let $x:=x_a=x_b$. Thus $a\vee b\leq x$. Then $x\in X_{a\vee b}\cap P$, which is a contradiction. Hence $a\vee b=\bigwedge X_{a\vee b}\in\B(\A)$. Let now $a,b\in\B(\A)$ be such that $a\wedge b$ exists in $A$. By Lemma \ref{lem: aux 1}, we have
\[
\bigwedge X_{a\wedge b} = \bigwedge(X_a\cup X_b) = \bigwedge X_a\wedge\bigwedge X_b = a\wedge b.
\]
Then, it follows by Proposition \ref{prop: meet of atoms are booleans} that $a\wedge b=\bigwedge X_{a\wedge b}\in\B(\A)$. Therefore, $\B(\A)$ is a nearlattice subalgebra of $\A$.
\end{proof}

\begin{proposition} \label{prop: [a)_B(A) = [a)_A}
Let $\A$ be a finite distributive nearlattice. Let $a\in\B(\A)$ and $b\in A$. If $a\leq b$, then $b\in\B(\A)$. That is, $[a)_{\B(\A)}=[a)_A$.
\end{proposition}

\begin{proof}
Let $a\in \B(\A)$ and $b\in A$ be such that $a\leq b$. Since $a$ is a boolean element, we have $a=\bigwedge X_a$. We also have that $b\leq\bigwedge X_b$. Suppose that $\bigwedge X_b\nleq b$. Then, there is a prime ideal $P$ such that $b\in P$ and $\bigwedge X_b\notin P$. Thus, $a\in P$ and $X_b\cap P=\emptyset$. Since $\bigwedge X_a\in P$ and $P$ is prime, we obtain that there is $x\in X_a\cap P$. Then, $x\notin X_b$. Since $x\in\At(\A)$ and $b\nleq x$, it follows that $b\vee x=1$. This is a contradiction because $b,x\in P$ and $P$ is a proper ideal. Hence $b=\bigwedge X_b\in\B(\A)$.
\end{proof}

A \emph{semi-boolean algebra} (\cite{Abb67}) is a join-semilattice with a top element such that every principal upset is a Boolean lattice. In \cite{Abb67} Abbot show that the semi-boolean algebras are in a one-to-one correspondence with the implication algebras (also known as Tarski algebras).

\begin{theorem}\label{theo: B(A) semi-boolean alg}
Let $\A=\langle A,\vee,1\rangle$ be a finite distributive nearlattice. Then, the nearlattice subalgebra $\langle\B(\A),\vee,1\rangle$ is a semi-boolean algebra.
\end{theorem}

\begin{proof}
Let $a\in\B(\A)$. By Proposition \ref{prop: [a)_B(A) = [a)_A}, we have $[a)_{\B(\A)}=[a)_A$, and since $\B(\A)$ is a nearlattice subalgebra of $\A$, it follows that $[a)_{\B(\A)}$ is a boolean lattice. Hence $\B(\A)$ is a semi-boolean algebra.
\end{proof}

\begin{proposition}
Let $A$ be a finite distributive nearlattice. The following are equivalent:
\begin{enumerate}[{\normalfont (1)}]
	\item $\bigwedge\At(\A)$ exists in $A$;
	\item There is $a\in A$ such that $\B(\A)=[a)$;
	\item $\B(\A)$ is a boolean lattice.
\end{enumerate}
\end{proposition}

\begin{proof}
$(1)\Rightarrow(2)$ Let $a:=\bigwedge\At(\A)$. It follows by Proposition \ref{prop: meet of atoms are booleans} that $a\in\B(\A)$. Let $b\in\B(\A)$. Then, $a=\bigwedge\At(\A)\leq\bigwedge X_b=b$. Thus $b\in[a)$. Hence $\B(\A)\subseteq[a)$. Now, by Proposition \ref{prop: [a)_B(A) = [a)_A}, we have $[a)=[a)_{\B(\A)}\subseteq\B(\A)$. Therefore, $\B(\A)=[a)$.

$(2)\Rightarrow(3)$ It follows straightforward by Theorem \ref{theo: B(A) semi-boolean alg}.

$(3)\Rightarrow(1)$ Let $a\in A$ be the least element of $\B(\A)$. Thus $[a)=[a)_{\B(\A)}=\B(\A)$. Since $\At(\A)\subseteq\B(\A)=[a)$, we obtain that $\bigwedge\At(\A)$ exists in $A$.
\end{proof}

\begin{corollary}\label{coro: B(L) boolean lattice}
If $\mathbf{L}$ is a finite distributive lattice, then $\B(\mathbf{L})$ is a boolean lattice. Moreover, for each $a\in\B(\L)$ and $x\in[a)$, we have $\neg_ax=\neg x\vee a$, where $\neg x$ denotes the complement of $x$ in $\B(\mathbf{L})$.
\end{corollary}

\begin{proof}
From the previous proposition, it is clear that $\B(\L)$ is a boolean lattice. Let $a\in\B(\L)$ and $x\in[a)$. Since $[a)$ is a boolean lattice, there is unique $\neg_a x \in [a)$ such that $x \vee \neg_a x = 1$ and $x \wedge \neg_a x = a$. As $a \leq x$, by Proposition \ref{prop: [a)_B(A) = [a)_A} we have $x \in\B(\L)$ and, there is $\neg x \in\B(\L)$ such that $x \vee \neg x = 1$ and $x \wedge \neg x = c$, where $c \in A$ is the least element of $\B(\L)$. Consider the element $\neg x \vee a \in [a)$. So, $(\neg x \vee a) \vee x = (\neg x \vee x) \vee a = 1 \vee a = 1$ and $(\neg x \vee a) \wedge x = (\neg x \wedge x) \vee (a \wedge x) = c \vee a = a$. Hence, we have $\neg_a x = \neg x \vee a$.
\end{proof}

\section{Annihilators}\label{sec:annihilators}

Let $\A=\langle A,\vee,1\rangle$ be a distributive nearlattice. For every $a\in A$, let $a^\top=\{x\in A: a\vee x=1\}$ (see \cite{CaCe15,ChaKo07}). It follows that $a^\top\in\Fi(\A)$, for all $a\in A$. We will say that an element $a\in A$ is \emph{dense} if $a^\top=\{1\}$. Denote by $\D(A)$ the set of all dense elements of $A$. 

\begin{proposition}
Let $\A$ be a finite distributive nearlattice. The following are equivalent:
\begin{enumerate}[{\normalfont (1)}]
	\item $\bigwedge \At(\A)$ exists in $A$;
	\item There is $a\in A$ such that $\D(A)=(a]$;
	\item $\D(A)$ is not empty.
\end{enumerate}
\end{proposition}

\begin{proof}
$(1)\Rightarrow(2)$ Let $a:=\bigwedge\At(\A)$. Let $x\in a^\top$. So $x\vee a=1$. Suppose that $x<1$. Then, there is $b\in\At(\A)$ such that $x\leq b$. Thus $b=b\vee a=1$, which is a contradiction. Hence $x=1$. It follows that $a^\top=\{1\}$, and thus $a\in\D(A)$. We obtain that $(a]\subseteq\D(A)$. Let now $b\in\D(A)$. So $b^\top=\{1\}$. Let $x\in\At(\A)$. Since $x\notin b^\top$, we have $b\vee x\neq 1$. Thus $b\leq x$. Then $b\leq\bigwedge\At(\A)=a$. Hence $b\in(a]$. Therefore, $\D(A)\subseteq(a]$.

$(2)\Rightarrow(3)$ It is immediate.

$(3)\Rightarrow(1)$ Let us show that $\At(\A)$ has a lower bound. As $\D(A) \neq \emptyset$, there is $a \in A$ such that $a^{\top} = \{1\}$. Let $x \in \At(\A)$ and suppose that $a \nleq x$.  Thus, $x \vee a = 1$. Then $x \in a^{\top} = \{1\}$, which is a contradiction because $x \in \At(\A)$. Hence $a \leq x$, for all $x \in\At(\A)$. Therefore, $\bigwedge \At(\A)$ exists in $A$.
\end{proof}

\begin{proposition} \label{prop: injective ()^top + b^top = a^top}
Let $\A$ be a finite distributive nearlattice. The following properties are satisfied:
\begin{enumerate}
\item Let $a,b\in \B(\A)$. Then, $a \leq b$ if and only if $a^{\top} \subseteq b^{\top}$.
\item For each $a \in A$, we have $a^\top=\left(\bigwedge X_a\right)^\top$.
\end{enumerate}
\end{proposition}

\begin{proof}
$(1)$ It is easy to show that $a \leq b$ implies $a^{\top} \subseteq b^{\top}$. Conversely, assume that $a^{\top} \subseteq b^{\top}$. Let $x\in X_b$. Then $x \in \At(\A)$ and $b \vee x = x \neq 1$, i.e., $x \notin b^{\top}$. So, $x \notin a^{\top}$. Since $a\vee x\neq 1$ and $x$ is dual atom, we have  $a\leq x$. Thus $x \in X_{a}$. Thus $X_{b} \subseteq X_{a}$. Therefore, $a = \bigwedge X_{a} \leq \bigwedge X_{b}=b$.

$(2)$ Let $a \in A$ and  $b := \bigwedge X_{a}$. By Proposition \ref{prop: meet of atoms are booleans}, $b \in \B(\A)$. Since $a \leq b$ we have $a^{\top} \subseteq b^{\top}$. Let $x \in b^{\top}$. Thus, 
\begin{equation*}
1=b\vee x=\left( \bigwedge X_{a} \right) \vee x=\bigwedge \{y \vee x \colon y \in X_{a} \}. 
\end{equation*} 
Then $y \vee x = 1$, for all $y\in X_{a}$. Suppose that $a \vee x <1$. So, there exists $y \in \At(\A)$ such that $a \vee x \leq y<1$. Thus $a \leq y$ and $x\leq y$. It follows that $y\in X_{a}$. Then $y=y\vee x=1$, which is a contradiction. Hence, $a \vee x = 1$. Thus $x\in a^{\top}$. Then, we obtain that $b^{\top} \subseteq a^{\top}$. Therefore, $a^\top=b^\top=\left(\bigwedge X_a\right)^\top$.
\end{proof}

Given a finite distributive nearlattice $\A$ we define the map $\pi_A\colon A\to\B(\A)$ as follows: for every $a\in A$,
\[
\pi_A(a)=\bigwedge X_a.
\]
From Proposition \ref{prop: meet of atoms are booleans}, it follows that $\pi_A$ is well defined.

\begin{remark}\label{rem: basic properties of pi_A}
Notice that $a\leq\pi_A(a)$, for all $a\in A$. Moreover, for each $a\in  A$, $\pi_A(a)$ is the least boolean element $b$ such that $a\leq b$. Thus, we also have that $\pi_A(a)=a$, for all $a\in\B(\A)$.
\end{remark}

\begin{proposition}
Let $\A$ be a finite distributive nearlattice. Then, the map $\pi_A\colon A\to\B(\A)$ is an onto N-homomorphism.
\end{proposition}

\begin{proof}
Let $a,b\in A$. By definition and Lemma \ref{lem: aux 1}, we have $\pi_A(a\vee b)=\bigwedge X_{a\vee b}=\bigwedge\left(X_a\cap X_b\right)$. Since $X_a\cap X_b\subseteq X_a,X_b$, it follows that $\pi_A(a)\vee\pi_A(b)\leq \pi_A(a\vee b)$. By Remark \ref{rem: basic properties of pi_A}, we have $a\leq\pi_A(a)$ and $b\leq\pi_A(b)$. Thus $a\vee b\leq\pi_A(a)\vee\pi_A(b)$. Then, since $\pi_A(a)\vee\pi_A(b)$ is a boolean element, it follows by Remark \ref{rem: basic properties of pi_A} that $\pi_A(a\vee b)\leq \pi_A(a)\vee\pi_A(b)$. Therefore, $\pi_A(a\vee b)=\pi_A(a)\vee\pi_A(b)$.

Now assume that $a\wedge b$ exists in $\A$. Then,
\[
\pi_A(a\wedge b)=\bigwedge X_{a\wedge b}=\bigwedge\left(X_a\cup X_b\right)=\bigwedge X_a\wedge\bigwedge X_b=\pi_A(a)\wedge\pi_A(b).
\]
Finally, $\pi_\A(1)=\bigwedge X_1=\bigwedge\emptyset=1$.
\end{proof}

It is well known that if $\L$ is a bounded distributive lattice, then the subset $\comp(\L)$ of all complemented elements of $L$ form a boolean algebra. Given that in a distributive nearlattice may not exists the least element, we generalize the concept of complemented element as follows. Recall that $\veebar$ denotes the supremum in $\Fi(\A)$ and $a^\top\in\Fi(\A)$, for all $a\in A$.

\begin{definition}
Let $\A$ be a distributive nearlattice. An element $a\in A$ is said to be \emph{complemented} if $[a)\veebar a^\top=A$.
\end{definition}

We denote by $\C(\A)$ the set of all complemented elements of $\A$.

\begin{proposition}\label{prop: G(A) = C(A)}
If $\A$ is a bounded distributive lattice, then $\comp(\A)=\C(\A)$.
\end{proposition}

\begin{proof}
Let $a\in \C(\A)$. So $[a)\veebar a^\top=A$. Let $0$ be the least element of $A$. Since $0\in[a)\veebar a^\top$, there is $b\in a^\top$ such that $a\wedge b=0$. Thus, we have $a\vee b=1$ and $a\wedge b=0$. Then, $b$ is the complement of $a$ in $A$. Hence $a\in\comp(A)$.

Let now $a\in\comp(A)$. Let $a^*\in A$ be the complement of $a$. Thus $a\wedge a^*=0$ and $a\vee a^*=1$. Let $x\in A$. Then $(a\vee x)\wedge(a^*\vee x)=(a\wedge a^*)\vee x=x$. Since $x\vee a\in[a)$ and $a^*\vee x\in a^\top$, it follows that $x\in[a)\veebar a^\top$. Hence $[a)\veebar a^\top=A$. Therefore, $a\in\C(\A)$.
\end{proof}

\begin{proposition}
Let $\A=\langle A,\vee,1\rangle$ be a distributive nearlattice. Then $\langle\C(\A),\vee,1\rangle$ is a nearlattice subalgebra of $\A$.
\end{proposition}

\begin{proof}
Let $a,b\in\C(\A)$. Thus $[a)\veebar a^\top=[b)\veebar b^\top=A$. Then,
\[
[a\vee b)\veebar (a\vee b)^\top= \left([a)\cap[b)\right)\veebar(a\vee b)^\top =\left([a)\veebar(a\vee b)^\top\right)\cap\left([b)\veebar(a\vee b)^\top\right).
\]
Since $a\leq a\vee b$, we have $a^\top\subseteq(a\vee b)^\top$. Thus $A=[a)\veebar a^\top\subseteq[a)\veebar(a\vee b)^\top$. Then $A=[a)\veebar(a\vee b)^\top$. Analogously, we have $A=[b)\veebar(a\vee b)^\top$. It follows $[a\vee b)\veebar (a\vee b)^\top=A$. Hence, $a\vee b\in\C(\A)$.

Assume now that $a\wedge b$ exists in $A$. Then,
\begin{align*}
[a\wedge b)\veebar(a\wedge b)^\top &= ([a)\veebar[b))\veebar\left(a^\top\cap b^\top\right)\\
 &= \left([a)\veebar [b)\veebar a^\top\right)\cap\left([a)\veebar[b)\veebar b^\top\right)
 = A.
\end{align*}
Hence $a\wedge b\in \C(\A)$. Therefore, $\langle\C(\A),\vee,1\rangle$ is a nearlattice subalgebra of $\A$.
\end{proof}

\begin{proposition}
Let $\A=\langle A,\vee,1\rangle$ be a distributive nearlattice. Then $\langle \C(\A),\vee,1\rangle$ is a semi-boolean algebra.
\end{proposition}

\begin{proof}
By the previous proposition we know that $\langle \C(\A),\vee,1\rangle$ is a distributive nearlattice. Thus, $[a)_{\C(\A)}$ is a bounded distributive lattice, for all $a\in\C(\A)$. Let $a\in\C(\A)$. Let us show that each element of the lattice $[a)_{\C(\A)}$ has a complement. Let $b\in[a)_{\C(\A)}$. So $a\leq b$ and $b\in\C(\A)$. Then $[b)\veebar b^\top=A$. Given that $a\in[b)\veebar b^\top$, there exist $x\in[b)$ and $y\in b^\top$ such that $a=x\wedge y$. Thus $a\leq y$ and $y\vee b=1$. Notice that $a\leq b,y$. Then $a\leq b\wedge y$. Since $b\leq x$, it follows that $b\wedge y\leq x\wedge y=a$. Hence $a=b\wedge y$. Now we show that $y\in\C(\A)$. Since $a=b\wedge y$, we obtain that $[a)=[b)\veebar[y)$ and $a^\top= b^\top\cap y^\top$. Moreover, since $b\in y^\top$, it follows that $[b)\subseteq y^\top$. Then, we have
\begin{align*}
A &= [a)\veebar a^\top = \left([b)\veebar[y)\right)\veebar\left(b^\top\cap y^\top\right)
=
\left([b)\veebar[y)\veebar b^\top\right)\cap\left([b)\veebar[y)\veebar y^\top\right)\\
&=
[y)\veebar y^\top.
\end{align*}
Hence, we obtain that $y\in[a)_{\C(\A)}$, $b\vee y=1$ and $b\wedge y=a$. That is, $y$ is the complement of $b$ in $[a)_{\C(\A)}$. Then $[a)_{\C(\A)}$ is a boolean algebra. Therefore, $\C(\A)$ is a semi-boolean algebra.
\end{proof}

\begin{proposition}[\cite{Ca19}]
Let $\A$ be a distributive nearlattice. Then, $\A$ is a semi-boolean algebra if and only if $A=\C(\A)$.
\end{proposition}

\begin{proposition}
Let $\A$ be a finite distributive nearlattice. Then, $\A$ is a semi-boolean algebra if and only if $\B(\A)=\C(\A)$.
\end{proposition}

\begin{proof}
Assume that $\A$ is a semi-boolean algebra. Then, $A=\B(\A)$ and $A=\C(\A)$. Hence $\B(\A)=\C(\A)$. Conversely, assume that $\B(\A)=\C(\A)$. Let $a\in A$. Then, by Proposition \ref{prop: injective ()^top + b^top = a^top}, $a^\top=b^\top$, where $b=\bigwedge X_a$. Thus, $b\in\B(\A)=\C(\A)$ and $a\leq b$. Then  $A=[b)\veebar b^\top\subseteq[a)\veebar b^\top=[a)\veebar a^\top$. Thus $a\in\C(\A)$. Hence $A=\C(\A)$. Therefore, $\A$ is a semi-boolean algebra.
\end{proof}

\section{Connection with the free distributive lattice extension}\label{sec: free extension}

\begin{definition}
Let $\A$ be a distributive nearlattice. A pair $\langle L(\A), e \rangle$, where $L(\A)$ is a bounded distributive lattice and $e \colon A \to L(\A)$ is an N-embedding, is said to be a \emph{free distributive lattice extension} of $\A$ if $e[A]$ is finitely meet-dense in $L(\A)$ and the following universal property holds: for every bounded distributive lattice $\M$ and every N-homomorphism $h \colon A \to M$, there exists a unique lattice homomorphism $\widehat{h} \colon L(\A) \to M$ such that $h= \widehat{h} \circ e$.
\end{definition}

\begin{lemma} \label{lemma1:free extension}
Let $\A$ be a distributive nearlattice and let $\langle L(\A), e \rangle$ be the free distributive lattice extension of $\A$. The following properties are satisfied:
\begin{enumerate}
\item $e \left[ [a) \right] = \left[ e(a) \right)$.
\item $\Fig_{L(\A)} \left( e \left[ a^{\top} \right] \right) = e(a)^{\top}$.
\end{enumerate}
\end{lemma}

\begin{proof}
$(1)$ It is clear that $e\left[[a)\right]\subseteq [e(a))$. Conversely, if $u \in \left[ e(a) \right)$, then $e(a) \leq u$. Since $e[A]$ is finitely meet-dense in $L(\A)$, there exist $x_{1}, \hdots, x_{n} \in A$ such that $u=e(x_{1}) \wedge \hdots \wedge e(x_{n})$. Then $e(a) \leq e(x_{i})$, for all $i \in \{1, \hdots, n\}$. Since $e$ is injective, we have $a \leq x_{i}$,  $\forall i \in \{1, \hdots, n\}$. Then, there exists $y:=x_{1} \wedge \hdots \wedge x_{n}\in[a)$ such that $e(y)=u$. Hence $u \in e \left[ [a) \right]$. Therefore, $e \left[ [a) \right] = \left[ e(a) \right)$.

$(2)$ We know that $e(a)^\top$ is a filter of $L(\A)$, and it is clear that $e\left[a^\top\right]\subseteq e(a)^\top$. Then, we have $\Fig_{L(\A)}\left(e\left[a^\top\right]\right)\subseteq e(a)^\top$. Let now $u\in e(a)^\top$. So $u\vee e(a)=1_{L(\A)}$. Let $a_1,\dots,a_n\in A$ be such that $u=e(a_1)\wedge\dots\wedge e(a_n)$. Thus $1_{L(\A)}=u\vee e(a)=(e(a_1)\vee e(a))\wedge\dots\wedge(e(a_n)\vee e(a))$. Then $1_{L(\A)}=e(a_i\vee a)$ for all $i$. Hence $a_i\vee a=1$ for all $i$. That is, $a_1,\dots,a_n\in a^\top$. Thus, $e(a_1),\dots,e(a_n)\in e\left[a^\top\right]$. Then $u=e(a_1)\wedge\dots\wedge e(a_n)\in\Fig_{L(\A)}\left(e\left[a^\top\right]\right)$. Hence, $\Fig_{L(\A)}\left(e\left[a^\top\right]\right)=e(a)^\top$.
\end{proof}

\begin{proposition}\label{prop: connection between elements of A y L_A}
Let $\A$ be a finite distributive nearlattice and $\langle L(\A),e\rangle$ its free distributive nearlattice extension. Then,
\begin{enumerate}[{\normalfont (1)}]
	\item $e\left[\At(\A)\right]=\At(L(\A))$;
	\item $e\left[\B(\A)\right]\subseteq\B(L(\A))$;
	\item $e\left[\C(\A)\right]\subseteq\comp(L(\A))$.
\end{enumerate}
\end{proposition}

\begin{proof}
(1) Let $a\in\At(\A)$. Let $u\in L(\A)$ be such that $e(a)\leq u\leq 1_{L(\A)}$ (recall that $1_{L(\A)}=e(1)$). Thus, there are $a_1,\dots,a_n\in  A$ such that $u=e(a_1)\wedge\dots\wedge e(a_n)$. Since $e(a)\leq u\leq e(a_i)$ for all $i$, we obtain that $a\leq a_i\leq 1$ for all $i$. Then, for every $i$, $a=a_i$ or $a_i=1$. If $a_i=1$ for all $i$, then $u=1_{L(\A)}$. Otherwise, there is $i\in\{1,\dots,n\}$ such that $a_i=a$. Thus, $e(a)\leq u\leq e(a_i)=e(a)$. Then $u=e(a)$. Hence $e(a)\in\At(L(\A))$. Let now $u\in\At(L(\A))$. There are $a_1,\dots,a_n\in A$ such that $u=e(a_1)\wedge\dots\wedge e(a_n)$. Since $u\neq 1_{L(\A)}$, it follows that there is $i$, such that $a_i\neq 1$. Then, since $u\leq e(a_i)<1_{L(\A)}$, we have $u=e(a_i)$. Moreover, it is straightforward to show directly that $a_i\in\At(\A)$. Hence $u\in e\left[\At(\A)\right]$.

(2) Let $a\in\B(\A)$. It is clear that $\langle[e(a)),\wedge,\vee,e(a),1_{L(\A)}\rangle$ is a bounded distributive sublattice of $L(\A)$. We need only to prove that every element of $[e(a))$ is complemented. Recall that $[e(a))=e\left[[a)\right]$. Thus, let $e(x)\in[e(a))$ with $x\in[a)$. Since $[a)$ is a boolean lattice, there is $\neg_{a} x \in [a)$ such that $x \vee \neg_{a} x = 1$ and $x \wedge_{a} \neg_{a} x = a$. Then $1_{L(\A)} = e \left( x \vee \neg_{a} x \right) = e(x) \vee e \left( \neg_{a} x \right)$ and $e(a) = e \left( x \wedge_{a} \neg_{a} x \right) = e(x) \wedge e \left( \neg_{a} x \right)$. Thus, $e \left( \neg_{a} x \right)$ is the complement of $e(x)$ in $\left[ e(a) \right)$, i.e., $e \left( \neg_{a} x \right) = \neg_{e(a)} e(x)$. Therefore, $e(a)\in\B(L(\A))$.

(3) Let $a\in\C(\A)$. So $[a)\veebar a^\top=A$. We need to prove that $e(a)\in\comp(L(\A))$. From Proposition \ref{prop: G(A) = C(A)}, it is equivalent to show that $[e(a))\veebar e(a)^\top =L(\A)$ (here $\veebar$ is the join of $\Fi(L(\A))$). Since $[a)\veebar a^\top=A$, it follows that
\[
\Fig_{L(\A)}\left(e\left[[a)\veebar a^\top\right]\right)=\Fig_{L(\A)}\left(e[A]\right)=L(\A).
\]
By \cite[Thm. 3.3]{CeCa16}, we know that the map $\Phi\colon\Fi(\A)\to\Fi(L(\A))$ defined by $\Phi(F)=\Fig_{L(\A)}\left(e[F]\right)$ for every $F\in\Fi(\A)$, is a lattice isomorphism. Then,
\begin{equation*}
\begin{split}
\Fig_{L(\A)}\left(e\left[[a)\veebar a^\top\right]\right) &= \Phi\left([a)\veebar a^\top\right) 
=  \Phi\left([a)\right)\veebar\Phi\left(a^\top\right)\\
&= \Fig_{L(\A)}\left(e\left[[a)\right]\right)\veebar\Fig_{L(\A)}\left(e\left[a^\top\right]\right)\\
&= [e(a))\veebar \Fig_{L(\A)}\left(e\left[a^\top\right]\right).
\end{split}
\end{equation*}
Hence, by Lemma \ref{lemma1:free extension}, we have
\[
L(\A)=\Fig_{L(\A)}\left(e\left[[a)\veebar a^\top\right]\right) = [e(a))\veebar e(a)^\top.
\]
Therefore, we have shown that $e(a)\in\C(L(\A))$.
\end{proof}

\begin{remark} \label{remark:2 items}
If $\A$ is a finite distributive nearlattice, then $L(\A)$ is a finite distributive lattice. Hence, by Corolary \ref{coro: B(L) boolean lattice}, it follows that $\B\left( L(\A) \right)$ is a boolean lattice and for all $a\in \B(\A)$, $e \left( \neg_{a} x \right) = \neg e(x) \vee e(a)$ for all $x\in[a)$. 
\end{remark}

\begin{theorem}
Let $\A$ be a finite distributive nearlattice and  $\langle L(\A), e \rangle$ its free distributive lattice extension. Then $\pi_{L(\A)} \circ e = e \circ \pi_{A}$, i.e., the following diagram commutes:
\begin{displaymath}
\begin{tabular}{ccc}
\xymatrix{
A \ar[r]^-{\pi_{A}} \ar[d]_-{e} & \B({\bf{A}}) \ar[d]^-{e} \\
L({\bf{A}}) \ar[r]_-{\pi_{L({\bf{A}})}} & \B \left( L({\bf{A}}) \right)
}
\end{tabular}
\end{displaymath}
\end{theorem}

\begin{proof}
Let $a \in A$. By Proposition \ref{prop: connection between elements of A y L_A}, we obtain that 
\begin{align*}
\pi_{L(\A)}\left(e(a)\right) &= \bigwedge\{z\in\At(L(\A)): e(a)\leq z\}\\
&=
\bigwedge\{e(x): x\in\At(\A), \, e(a)\leq e(x)\}\\
&=
e\left[\bigwedge\{x\in\At(\A): a\leq x\}\right] = e\left(\pi_\A(a)\right).
\qedhere
\end{align*}
\end{proof}

\section{A discrete representation}\label{sec:representation}

In this section we develop a representation for the class of finite distributive nearlattices through certain ordered structures. This representation is a nice generalization of that given by Birkhoff for finite distributive lattice through finite posets (\cite{Bi37}).

\begin{definition}
Let $\A$ be a distributive nearlattice. An element $a\in A$ is said to be \emph{meet-irreducible} (or simply \emph{irreducible}) if for all $x,y\in A$ such that $x\wedge y$ exists in $A$, $a=x\wedge y$ implies $a=x$ or $a=y$. 
\end{definition}

Let us denote by $\Irr(\A)$ the set of all irreducible elements of $\A$. Notice that $\At(\A)\subseteq\Irr(\A)$.

\begin{proposition}\label{prop: irreducible = prime}
Let $\A$ be a distributive nearlattice and $a\in A$. Then, $a$ is irreducible if and only if for all $x,y\in A$ such that $x\wedge y$ exists, $x\wedge y\leq a$ implies $x\leq a$ or $y\leq a$.
\end{proposition}

\begin{theorem}
Let $\A$ be a finite distributive nearlattice. Then, for every $a\in A$,
\[
a=\bigwedge\{x\in\Irr(\A): a\leq x\}.
\]
\end{theorem}

\begin{proof}
Let $S=\{a\in A: a\neq \bigwedge\{x\in\Irr(\A): a\leq x\}\}$. We suppose by contradiction that $S\neq\emptyset$. Then, since $A$ is finite, $S$ has a maximal element $m$. Thus $m\notin\Irr(\A)$. Then, there exist $a,b\in A$ such that $m=a\wedge b$, $m<a$ and $m<b$. Since $m$ is a maximal element of $S$, it follows that $a,b\notin S$. Thus $a=\bigwedge\{x\in\Irr(\A): a\leq x\}$ and $b=\bigwedge\{\Irr(\A): b\leq x\}$. Hence, by Proposition \ref{prop: irreducible = prime}, we have
\begin{align*}
m &= \left(\bigwedge\{x\in\Irr(\A): a\leq x\}\right)\wedge\left(\bigwedge\{x\in\Irr(\A): b\leq x\}\right)\\
&=
\bigwedge\{x\in\Irr(\A): m\leq x\}.
\end{align*}
Thus, we obtain that $m\notin S$, which is a contradiction. Therefore, $S=\emptyset$.
\end{proof}

\begin{proposition}\label{prop: Irr(A)=Irr(L_A)}
Let $\A$ be a finite distributive nearlattice and $\langle L(\A),e\rangle$ its free distributive lattice extension. Then, $e\left[\Irr(\A)\right]=\Irr(L(\A))$.
\end{proposition}

\begin{proof}
Let $x\in\Irr(\A)$. Let $u,v\in L(\A)$ be such that $e(x)=u\wedge v$. There are $a_1,\dots,a_n,b_1,\dots,b_m\in A$ such that $u=e(a_1)\wedge\dots\wedge e(a_n)$ and $v=e(b_1)\wedge\dots\wedge e(b_m)$. Thus, we have that $e(x)\leq e(a_i),e(b_j)$, for all $i=1,\dots,n$ and $j=1,\dots,m$. Then $x\leq a_i,b_j$, for all $i,j$, which implies that there exists $a_1\wedge\dots\wedge a_n\wedge b_1\wedge\dots\wedge b_m$. Then $e(x)=e(a_1\wedge\dots\wedge a_n\wedge b_1\wedge\dots\wedge b_m)$. So $x=a_1\wedge\dots\wedge a_n\wedge b_1\wedge\dots\wedge b_m$. Since $x$ is irreducible, it follows that $x=a_1\wedge\dots\wedge a_n$ or $x=b_1\wedge\dots\wedge b_m$. Then $e(x)=u$ or $e(x)=v$. Hence $e(x)\in\Irr(L(\A))$. Therefore $e\left[\Irr(\A)\right]\subseteq\Irr(L(\A))$. Now it is straightforward to show the inclusion $\Irr(L(\A))\subseteq e\left[\Irr(\A)\right]$. 
\end{proof}

From now on, given a poset $\langle X,\leq\rangle$, $\Down(X)$ will denote the collection of all downsets of $X$ and let us consider the bounded distributive lattice $\langle\Down(X),\cap,\cup,\emptyset,X\rangle$.

Let $\A$ be a finite distributive nearlattice and $\langle L(\A),e\rangle$ its free distributive lattice extension. Let us consider $\Irr(\A)$ ($\Irr(L(\A))$) as a sub-poset of $A$ ($L(\A)$), that is, $x\leq y$ if and only if $x\vee y=y$, for all $x,y\in\Irr(\A)$. Since $L(\A)$ is a finite distributive lattice, it follows that the map $\alpha\colon L(\A)\to \Down(\Irr(L(\A)))$ defined by $\alpha(u)=\{z\in\Irr(L(\A)): u\nleq z\}$ is an isomorphism. Hence, by Proposition \ref{prop: Irr(A)=Irr(L_A)}, we have that the map $\widehat{\alpha}\colon L(\A)\to\Down(\Irr(\A))$ given by $\widehat{\alpha}(u)=\{x\in\Irr(\A): u\nleq e(x)\}$ is an isomorphism. Therefore, we have the following.

\begin{proposition}\label{prop: free extension of A}
For every finite distributive nearlattice $\A$, we have that $\langle\Down(\Irr(\A)),\widehat{e}\rangle$ is the free distributive lattice extension of $\A$, where $\widehat{e}\colon A\to\Down(\Irr(\A))$ is given by $\widehat{e}(a)=(\widehat{\alpha}\circ e)(a)=\{x\in\Irr(\A): a\nleq x\}$.
\end{proposition}

\begin{definition}
A \emph{DN-structure} is a pair $\langle X,\gamma\rangle$ such that $X$ is a poset and $\gamma\colon \Down(X)\to\{0,1\}$ is a map satisfying the following:
\begin{enumerate}[(S1)]
	\item $\gamma(X)=1$;
	\item $\gamma([x)^c)=1$, for all $x\in X$;
	\item for all $U,V\in\Down(X)$, $U\subseteq V$ implies $\gamma(U)\leq\gamma(V)$.
\end{enumerate}
\end{definition}

We say that a DN-structure $\langle X,\gamma\rangle$ is \emph{finite} if the poset $X$ is finite. Let $\langle X,\gamma\rangle$ be a DN-structure. We define
\[
\N(X):=\{U\in\Down(X): \gamma(U)=1\}.
\]

\begin{proposition}
Let $\langle X,\gamma\rangle$ be a finite DN-structure. Then $\langle\N(X),\cup,X\rangle$ is a distributive nearlattice and $\langle X,\leq\rangle\cong\langle \Irr(\N(X)),\subseteq\rangle$.
\end{proposition}

\begin{proof}
First let us show that $\N(X)$ is closed under $\cup$. Let $U_1,U_2\in\N(X)$. Thus $\gamma(U_1)=\gamma(U_2)=1$. Then, since $U_1\subseteq U_1\cup U_2$, it follows by (S3) that $\gamma(U_1\cup U_2)=1$. Hence $U_1\sqcup U_2\in\N(X)$

Now let $U_1,U_2,V\in\N(X)$ be such that $V\subseteq U_1,U_2$. Thus $V\subseteq U_1\cap U_2$. By (S3), we have $U_1\cap U_2\in\N(X)$. Then, $U_1\cap U_2$ is the meet of $U_1$ and $U_2$ in $[V)_{\N(X)}$. Hence $\langle[V)_{\N(X)},\cap,\cup,U,X\rangle$ is a bounded distributive lattice. Therefore, $\langle\N(X),\cup,X\rangle$ is a distributive nearlattice.

It is well-known that $\Irr(\Down(X))=\{[x)^c: x\in X\}$, and thus $\langle\Irr(\Down(X)),\subseteq\rangle\cong\langle X,\leq\rangle$. Let us prove that $\Irr(\N(X))=\Irr(\Down(X))$. By (S2), we have $\Irr(\Down(X))\subseteq\Irr(\N(X))$. Let now $U\in\Irr(\N(X))$. Since $U^c$ is a finite upset of $X$, we have that $U^c=[x_1)\cup\dots\cup[x_n)$, for some $x_1,\dots,x_n\in U^c$. Thus $U=[x_1)^c\cap\dots\cap[x_n)^c$. Then, since $U\in\Irr(\N(X)$, we obtain that $U=[x_i)^c$, for some $i=1,\dots,n$. Hence $\Irr(\N(X))\subseteq\Irr(\Down(X))$. Therefore, we have that $\langle \Irr(\N(X)),\subseteq\rangle\cong\langle X,\leq\rangle$.
\end{proof}

Let $\langle A,\vee,1\rangle$ be a finite distributive nearlattice. Let $\S(A)=\langle\Irr(\A),\leq\rangle$. 

\begin{proposition}
Let $\langle A,\vee,1\rangle$ be a finite distributive nearlattice. Then, the pair $\langle\S(A),\gamma_A\rangle$ is a DN-structure, where $\gamma_A\colon\Down(\S(A))\to\{0,1\}$ is the map defined by $\gamma_A(U)=1$ if and only if  $\bigwedge (\S(A)\setminus U)$ exists in $A$.
\end{proposition}

\begin{proof}
It is straightforward to show directly that the map $\gamma_A$ satisfies conditions (S1)--(S3).
\end{proof}

Given a finite distributive nearlattice $\A$, we have that $\langle\N(\S(A)),\cup,\S(A)\rangle$ is a finite distributive nearlattice, where $\N(\S(A))=\{U\in\Down(\S(A)): \gamma_A(U)=1\}$.

\begin{theorem}[Discrete representation]
Let $\A=\langle A,\vee,1\rangle$ be a finite distributive nearlattice. Then, $\A\cong\N(\S(A))$. 
\end{theorem}

\begin{proof}
From Proposition \ref{prop: free extension of A}, we have that the map $\widehat{e}\colon A\to \Down(\S(A))$ is an N-embedding, where $\widehat{e}(a)=\{x\in\Irr(\A): a\nleq x\}$, for every $a\in A$. Thus, $\A\cong\widehat{e}[A]$. Let us show that $\widehat{e}[A]=\N(\S(A))$. Then,
\begin{align*}
U\in\N(\S(A)) &\iff U\in\Down(\S(A)) \text{ and } \gamma_A(U)=1\\
&\iff
U\in\Down(\S(A)) \text{ and } \bigwedge \S(A)\setminus U \text{ exists in } A\\
&\iff
\text{there exists } a\in A \text{ such that } U=\{x\in\Irr(\A): a\nleq x\}\\
&\iff
\text{there exists } a\in A \text{ such that } U=\widehat{e}(a)\\
&\iff
U\in\widehat{e}[A].
\end{align*}
Hence, $\A\cong\widehat{e}[A]=\N(\S(A))$.
\end{proof}

\begin{remark}
Let $\A$ be a finite distributive nearlattice. If $\A$ is in fact a lattice, then it follows that $\N(\S(\A))=\Down(\S(\A))$. Thus $\A\cong\Down(\S(\A))$. Moreover, it is clear that the finite posets are in  one-to-one correspondence with the DN-structures $\langle X,\gamma\rangle$ such that $\gamma(U)=1$, for all $U\in\Down(X)$. Therefore, from the representation above established we can obtain the representation given by Birkhoff for finite distributive lattices.
\end{remark}

\section*{Acknowledgments}

This work was partially supported by ANPCyT (Argentina) under the Grant PICT-2019-00882 and by CONICET (Argentina) under the Grant PIP 112-201501-00412. The first author was also partially supported by ANPCyT  under the Grant PICT-2019-00674 and by Universidad Nacional de La Pampa under the Grant P.I. No 78M, Res. 523/19.


\begin{thebibliography}{10}

\bibitem{Abb67}
J.~Abbott.
\newblock Semi-boolean algebra.
\newblock {\em Matemati{\v{c}}ki Vesnik}, 4(19):177--198, 1967.

\bibitem{ArKi11}
J.~Ara\'ujo and M.~Kinyon.
\newblock Independent axiom systems for nearlattices.
\newblock {\em Czech. Math. J.}, 61(4):975--992, 2011.

\bibitem{BaDwi74}
R.~Balbes and P.~Dwinger.
\newblock {\em Distributive lattices}.
\newblock University of Missouri Press, 1974.

\bibitem{Bi37}
G.~Birkhoff.
\newblock Rings of sets.
\newblock {\em Duke Math. J.}, 3(3):443--454, 1937.

\bibitem{Ca15}
I.~Calomino.
\newblock {\em Supremo \'algebra distributivas: una generalizaci\'on de las
  \'algebra de {Tarski}}.
\newblock PhD thesis, Universidad Nacional del Sur, 2015.

\bibitem{Ca19}
I.~Calomino.
\newblock Note on $\alpha$-filters in distributive nearlattices.
\newblock {\em Mathematica Bohemica}, 144(3):241--250, 2019.

\bibitem{CaCe15}
I.~Calomino and S.~Celani.
\newblock A note on annihilators in distributive nearlattices.
\newblock {\em Miskolc Math. Notes}, 16(1):65--78, 2015.

\bibitem{CaCeGo19}
I.~Calomino, S.~Celnai, and L.~J. Gonz\'alez.
\newblock Quasi-modal operators on distributive nearlattices.
\newblock {\em Rev. Un. Mat. Argentina}, In Press, 2019.

\bibitem{CaGo20}
I.~Calomino and L.~J. Gonz\'alez.
\newblock Remarks on normal distributive nearlattices.
\newblock {\em Quaest. Math.}, pages 1--12, 2020.

\bibitem{CeCa14}
S.~Celani and I.~Calomino.
\newblock Stone style duality for distributive nearlattices.
\newblock {\em Algebra universalis}, 71(2):127--153, 2014.

\bibitem{CeCa16}
S.~Celani and I.~Calomino.
\newblock On homomorphic images and the free distributive lattice extension of
  a distributive nearlattice.
\newblock {\em Rep. Math. Logic}, 51:57--73, 2016.

\bibitem{CeCa19}
S.~Celani and I.~Calomino.
\newblock Distributive nearlattices with a necessity modal operator.
\newblock {\em Math. Slovaca}, 69:35--52, 2019.

\bibitem{ChaHaKu07}
I.~Chajda, R.~Hala{\v{s}}, and J.~K{\"u}hr.
\newblock {\em Semilattice structures}, volume~30 of {\em Research and
  Exposition in Mathematics}.
\newblock Heldermann Verlag, Lemgo, 2007.

\bibitem{ChaKo07}
I.~Chajda and M.~Kola{\v{r}}{\'\i}k.
\newblock Ideals, congruences and annihilators on nearlattices.
\newblock {\em Acta Univ. Palacki. Olomuc. Fac. rer. nat. Mathematica},
  46(1):25--33, 2007.

\bibitem{ChaKo08}
I.~Chajda and M.~Kola{\v{r}}{\'\i}k.
\newblock Nearlattices.
\newblock {\em Discrete Math.}, 308(21):4906--4913, 2008.

\bibitem{DaPri02}
B.~Davey and H.~Priestley.
\newblock {\em Introduction to lattices and order}.
\newblock Cambridge University Press, 2002.

\bibitem{Go18}
L.~J. Gonz\'alez.
\newblock The logic of distributive nearlattices.
\newblock {\em Soft Computing}, 22(9):2797--2807, 2018.

\bibitem{Go19}
L.~J. Gonz\'alez.
\newblock Selfextensional logics with a distributive nearlattice term.
\newblock {\em Arch. Math. Logic}, 58:219--243, 2019.

\bibitem{GoCa19}
L.~J. Gonz\'alez and I.~Calomino.
\newblock A completion for distributive nearlattices.
\newblock {\em Algebra Universalis}, 80: 48, 2019.

\bibitem{Gra11}
G.~Gr{\"a}tzer.
\newblock {\em Lattice theory: foundation}.
\newblock Springer Science \& Business Media, 2011.

\bibitem{Ha06}
R.~Hala{\v{s}}.
\newblock Subdirectly irreducible distributive nearlattices.
\newblock {\em Miskolc Math. Notes}, 7:141--146, 2006.

\end{thebibliography}
\end{document}